\documentclass[10pt]{amsart}
\usepackage{amssymb,amsmath,txfonts,amsthm}
\usepackage{hyperref}
\usepackage{mathrsfs}
\newcommand\blfootnote[1]{%
  \begingroup
  \renewcommand\thefootnote{}\footnote{#1}%
  \addtocounter{footnote}{-1}%
  \endgroup
}
\newtheorem{theorem}{Theorem}
\newtheorem{prop}{Proposition}
\newtheorem{lemma}{Lemma}
\newtheorem*{remark}{Remark}
\newtheorem{claim}{Claim}
\newtheorem*{definition}{Definition}
\newtheorem{cor}{Corollary}

\def\Xint#1{\mathchoice
  {\XXint\displaystyle\textstyle{#1}}%
  {\XXint\textstyle\scriptstyle{#1}}%
  {\XXint\scriptstyle\scriptscriptstyle{#1}}%
  {\XXint\scriptscriptstyle\scriptscriptstyle{#1}}%
  \!\int}
\def\XXint#1#2#3{{\setbox0=\hbox{$#1{#2#3}{\int}$}
  \vcenter{\hbox{$#2#3$}}\kern-.5\wd0}}

\def\dashint{\Xint-}

\author{Gang Liu}
\address{Department of Mathematics\\University of California, Berkeley\\Berkeley, CA 94720}
\email{gangliu@math.berkeley.edu}
\title[tangent cone of K\"ahler manifolds]{On the tangent cone of K\"ahler manifolds with Ricci curvature lower bound}
\date{}
\begin{document}
\begin{abstract}
Let $X$ be the Gromov-Hausdorff limit of a sequence of pointed complete K\"ahler manifolds $(M^n_i, p_i)$ satisfying $Ric(M_i)\geq -(n-1)$ and the volume is noncollapsed. 
We prove that, there exists a Lie group isomorphic to $\mathbb{R}$, acting isometrically, on the tangent cone at each point of $X$. Moreover, the action is locally free on the cross section. This generalizes the metric cone theorem of Cheeger-Colding to the K\"ahler case. We also discuss some applications to complete K\"ahler manifolds with nonnegative bisectional curvature.\end{abstract}
\maketitle

\section{\bf{Introduction}}
\blfootnote{The author was partially supported by NSF grant DMS 1406593.}
In \cite{[CC1]}, Cheeger and Colding proved the important metric cone theorem: 
\begin{theorem}\label{cc}[Cheeger-Colding]
Let $(X, p_\infty)$ be the Gromov-Hausdorff limit of a sequence of pointed complete Riemannian manifolds $(M^m_i, p_i)$ with $Ric(M_i)\geq -(m-1)$ and noncollapsed volume. Consider a point $q\in X$ and a tangent cone $(X_q, q)$. Then there exists a compact metric length space $\Sigma$ so that $(X_q, q)$ is isometric to the warped product $\Sigma\times_{r^2}\mathbb{R}^+$.\end{theorem}
This theorem implies for complete manifolds, with nonnegative Ricci curvature and maximal volume growth, all tangent cones at infinity are metric cones. Also, the ideas behind the proof resolve many important problems: e.g., splitting theorem for Gromov-Hausdorff limit $X$, where $M^n_i\to X$, $Ric_{M_i}\geq -\epsilon_i$ and $\epsilon_i\to 0$.
See also \cite{[CC2]}\cite{[CC3]}\cite{[CC4]} for further developments.

Theorem \ref{cc} says $X_q$, as a metric space, has a symmetry from the radial direction. Thus we can reduce the geometry to the cross section.
In this paper, we are interested in generalizing theorem \ref{cc} to K\"ahler manifolds. We wonder whether there are extra symmetries for $X_q$. Observe that if $X_q$ is smooth (K\"ahler) apart from the vertex, the cross section is a Sasakian manifold \cite{[MSY]}. Thus, it has an additional symmetry induced by the reeb vector field. 

In this paper, we only assume the Ricci curvature lower bound and noncollapsed volume. In this case, it was proved in \cite{[CCT]} that the any tangent cone splits off even dimensional Euclidean factor. The main result in this paper is the following:
\begin{theorem}\label{thm1}
Let $(X, p_\infty)$ be the Gromov-Hausdorff limit of a sequence of pointed complete K\"ahler manifolds $(M^n_i, p_i)$ with $Ric(M_i)\geq -(n-1)$ and noncollapsed volume. Consider a point $q\in X$ and a tangent cone $(X_q, q)=\Sigma\times_{r^2}\mathbb{R}^+$. Then there exists a Lie group $\sigma_t$, isomorphic to $\mathbb{R}$, acting isometrically and locally freely on $\Sigma$. Obviously, we can extend the isometry $\sigma_t$ to $(X_q, q)$ which preserves the vertex $q$. 
\end{theorem}
\begin{remark}
This theorem is trivial when the Ricci curvature is bounded from two sides. The original motivation of theorem \ref{thm1} is to prove a conjecture of Yau which states that the ring of polynomial growth holomorphic functions is finitely generated, provided the manifold has nonnegative bisectional curvature. By using proposition \ref{prop-100}, in some sense, we can reduce the problem to the cross section. This could be further reduced by proposition \ref{thm-10}. \end{remark}

\begin{remark}
By taking the closure of $\sigma_t$ in the isometry group of $(X_q, q)$, we obtain an effective torus group acting isometrically on $(X_q, q)$.
\end{remark}
\begin{cor}\label{co1}
Let $(X, p_\infty)$ be a tangent cone at infinity of a complete K\"ahler manifold $(M^n, p)$, with $Ric\geq 0$ and noncollapsed volume. Then there is an effective torus isometry group acting on $(X, p_\infty)$ which commutes with the homothety map induced by $\nabla r^2$, where $r(x) = d(x, p_\infty)$.
\end{cor}

In \cite{[CT]}, Cheeger and Tian considered K\"ahler Ricci flat metric with Euclidean volume growth and quadratic curvature decay. In that case, the tangent cone at infinity is  K\"ahler apart from the vertex. They looked at the flow generated by $r\frac{\partial}{\partial r}$ and $J(r\frac{\partial}{\partial r})=\frac{\partial}{\partial\theta}$ on the tangent cone. Here $r$ is the distance to the vertex.
This is a $\tilde{\mathbb{C}^*}$ action, where $\tilde{\mathbb{C}^*}$ is the universal cover of $\mathbb{C^*} = \mathbb{C}\backslash \{0\}$. Corollary \ref{co1} says we still have this action, provided the K\"ahler manifold has nonnegative Ricci curvature and maximal volume growth.

Next we discuss applications to complete K\"ahler manifolds with nonnegative bisectional curvature and maximal volume growth.  Let $\{\mathcal{O}_d(M)\}$ be the ring of holomorphic functions with polynomial growth on $M$. We prove that these functions are all homogenous at infinity. We also study the behavior of these functions under the action of $\sigma_t$. As a corollary,  the dimension of the orders of $\{\mathcal{O}_d(M)\}$ over $\mathbb{Q}$, is no greater than, the dimension of the isometry group of $\Sigma$. In particular, if the dimension of isometry group is $1$ on the cross section, then the orders of holomorphic functions with polynomial growth are rationally related.

\medskip
\begin{center}
\bf  {\quad Acknowledgment}
\end{center}
The author thanks Professor Gang Tian for his encouragement and interest on this work during the visit to UC Berkeley on March 31, 2014. 
The author also thanks Professors John Lott, Xiaochun Rong and Jiaping Wang for valuable discussions. 

\section{\bf{Preliminary results}}
In this section, we collect some basic definitions and results required in the proof.
Let $(M^n_i, y_i, \rho_i)$ be a sequence of pointed complete Riemannian manifolds, where $y_i\in M^n_i$ and $\rho_i$ is the metric on $M^n_i$. By Gromov's compactness theorem, if $(M^n_i, y_i, \rho_i)$ have a uniform lower bound of the Ricci curvature, then a subsequence converges to some $(M_\infty, y_\infty, \rho_\infty)$ in the Gromov-Hausdorff topology. See \cite{[GLP]} for the definition and basic properties of Gromov-Hausdorff convergence.
\begin{definition}
Let $K_i\subset M^n_i\to K_\infty\subset M_\infty$ in the Gromov-Hausdorff topology. Assume $\{f_i\}_{i=1}^\infty$ are functions on $M^n_i$, $f_\infty$ is a function on $M_\infty$.  
$\Phi_i$ are $\epsilon_i$-Gromov-Hausdorff approximations, $\lim\limits_{i\to\infty} \epsilon_i = 0$. If $f_i\circ \Phi_i$ converges to $f_\infty$ uniformly, we say $f_i\to f_\infty$ uniformly over $K_i\to K_\infty$.
\end{definition}
 In many applications, $f_i$ are equicontinuous. The Arzela-Ascoli theorem applies to the case when the spaces are different.  When $(M_i^n, y_i, \rho_i)\to (M_\infty, y_\infty, \rho_\infty)$ in the Gromov-Hausdorff topology, any bounded, equicontinuous sequence of functions $f_i$ has a subsequence converging uniformly to some $f_\infty$ on $M_\infty$.

Let the complete pointed metric space $(M_\infty^m, y)$ be the Gromov-Hausdorff limit of a sequence of connected pointed Riemannian manifolds, $\{(M_i^n, p_i)\}$, with $Ric(M_i)\geq 0$. Here $M_\infty^m$ has Haudorff dimension $m$ with $m\leq n$. A tangent cone at $y\in M_\infty^m$ is a complete pointed Gromov-Hausdorff limit $((M_\infty)_y, d_\infty, y_\infty)$ of $\{(M_\infty, r_i^{-1}d, y)\}$, where $d, d_\infty$ are the metrics of $M_\infty, (M_\infty)_y$ respectively, $\{r_i\}$ is a positive sequence converging to $0$.

\begin{definition}
A point $y\in M_\infty$ is called regular, if there exists some $k$ so that every tangent cone at $y$ is isometric to $\mathbb{R}^k$. A point is called singular, if it is not regular.
\end{definition}
\begin{theorem}[Theorem $2.1$, \cite{[CC2]}]
Regular points are dense on $M_\infty$.
\end{theorem}

Let $M^n$ be a complete Riemannian manifold with nonnegative curvature, $p\in M$. The tangent cone at infinity is the Gromov-Hausdorff limit of $(M_i, p_i, g_i) = (M, p, r_i^{-2}g)$, for $r_i\to\infty$. We say $M$ is of maximal volume growth (Euclidean volume growth), if $\lim\limits_{r\to\infty}\frac{Vol(B(p, r))}{r^n}>0$.
A sequence of pointed manifolds $(M_i, p_i)$ with $Ric(M_i)\geq -(n-1)$ has noncollapsed volume, if there exists $v>0$ so that $vol(B(p_i, 1))> v$ for all $i$.

For a Lipschitz function $f$ on $M_\infty$, define a norm $||f||^2_{1, 2} = ||f||^2_{L^2}+\int_{M_\infty}|Lip f|^2$, where $$Lip(f, x) =\lim\sup\limits_{y\to x}\frac{|f(y)-f(x)|}{d(x, y)}.$$
In \cite{[Che]}, a Sobolev space $H_{1, 2}$ is defined by taking the closure of the norm $||\cdot||_{1, 2}$ for Lipschitz functions.

\emph{Condition (1)}:
$M_\infty$ satisfies the volume doubling property if for any $r>0$, $x\in M_\infty$,  $\nu_\infty(B(x, 2r))\leq 2^n\nu_\infty(B(x, r))$.

\emph{Condition (2)}:
$M_\infty$ satisfies the weak Poincare inequality if $$\int_{B(x, r)}|f - \overline{f}|^2 \leq C(n)r^2\int_{B(x, 2r)}|Lip f|^2$$ for all Lipschitz functions.
Here $\overline f$ is the average of $f$ on $B(x, r)$.

In theorem $6.7$ of \cite{[CC4]}, it was proved that if $M_\infty$ satisfies the $\nu$-rectifiability condition, condition (1) and condition (2), then there is a unique differential $df$ for $f\in H_{1, 2}$.
If $f$ is Lipschitz, $\int |Lipf|^2 = \int |df|^2$. Moreover, the $H_{1, 2}$ norm becomes an inner product. Therefore $H_{1, 2}$ is a Hilbert space. Then there exists a unique self-adjoint operator $\Delta$ on $M_\infty$ such that $$\int_{M_\infty} <df, dg> = \int _{M_\infty}<\Delta f, g>$$ for all Lipschitz functions on $M_\infty$ with compact support (Of course we can extend the functions to Sobolev spaces). See theorem $6.25$ of \cite{[CC4]}.

If $M_i\to M_{\infty}$ in the measured Gromov-Hausdorff sense and that the Ricci curvature is nonnegative for all $M_i$, then the $\nu$-rectifiability of $M_\infty$ was proved in theorem $5.5$ in \cite{[CC4]}.
By the volume comparison, Condition (1) obviously holds for $M_\infty$. Condition (2) also holds. See \cite{[X]} for a proof.

In \cite{[Di1]}\cite{[X]}, the following lemma was proved:
\begin{lemma}\label{lemma-10}
Suppose $M_i$ has nonnegative Ricci curvature and $M_i\to M_\infty$ in the measured Gromov-Hausdorff sense. Let $f_i$ be Lipschitz functions on $B(x_i, 2r)\subset M_i$ satisfying $\Delta f_i = 0$; $|f_i|\leq L, |\nabla f_i|\leq L$ for some constant $L$. Assume $x_i\to x_\infty$, $f_i\to f_\infty$ on $M_\infty$. Then $\Delta f_\infty = 0$ on $B(x_\infty, r)$.
\end{lemma}

\bigskip
\section{\bf{Proof of the main theorem}}
\begin{proof}
Through out the proof, we will denote by $\Phi(u_1,..., u_k|....)$ any nonnegative functions depending on $u_1,..., u_k$ and some additional parameters such that when these parameters are fixed, $$\lim\limits_{u_1,..., u_k\to 0}\Phi(u_1,..., u_k|...) = 0.$$ According to theorem \ref{cc}, $(X_q, q)$ is a metric cone. Say $(X_q, q) = (\Sigma\times_{r^2}\mathbb{R}^+, q)$. We may assume $(N_i^n, q_i)$ pointed converges in the Gromov-Hausdorff sense to $(X_q, q)$, where $(N_i^n, q_i)$ are K\"ahler, with $Ric(N_i)\geq-\frac{1}{i}$ and the volume is noncollapsed. Consider the geodesic annulus $A_i=B(q_i, 10)\backslash B(q_i, \frac{1}{3})$. By results in \cite{[CC1]} (explicitly, Corollary $4.42$,  Lemma $4.45$, Proposition $4.50$, Proposition $4.82$), there exist smooth functions $\rho_i$ on $N_i$ so that 
\begin{equation}\label{1}
\int_{A_i}|\nabla\rho_i-\nabla \frac{1}{2}r_i^2|^2 + |\nabla^2\rho_i-g_i|^2<\Phi(\frac{1}{i});
\end{equation}
\begin{equation}\label{2}
 |\nabla\rho_i|\leq C_1(n), |\rho_i-\frac{r_i^2}{2}|<\Phi(\frac{1}{i})
 \end{equation}  in $A_i$. Define \begin{equation}\label{3}X_i = J\nabla\rho_i\end{equation} and let $\sigma_{i, t}$ be the 
diffeomorphism generated by $X_i$. Clearly, $\sigma_{i, t}$ preserves the level set of $\rho_i$. 
Let $x\in B(q_i, 7)\backslash B(q_i, 4)$. Define functions \begin{equation}\label{4}F(x, t) = \int_{B(x, \frac{1}{2})}\int_{B(x, \frac{1}{2})}|d(y, z) - d(\sigma_{i, t}(y), \sigma_{i, t}(z))|^2d\sigma_{i, t}^*vol(y)d\sigma_{i, t}^*vol(z),\end{equation} \begin{equation}\label{5}G(x, t) = \int_{B(x, \frac{1}{2})}\int_{B(x, \frac{1}{2})}|d(y, z) - d(\sigma_{i, t}(y), \sigma_{i, t}(z))|^2dvol(y)dvol(z).\end{equation}
Here $\sigma_{i, t}^*vol$ is the pull back volume form. $|t|\leq \frac{1}{10C_1(n)}$, where $C_1(n)$ appears in (\ref{2}).
\begin{claim}\label{cl1}
For $|t|\leq \frac{1}{10C_1(n)}$ and $y, z\in B(x, \frac{1}{2})$, $\sigma_{i, t}(y)\in B(x, \frac{2}{3})$. In particular, if $l$ is a shortest geodesic connecting $\sigma_{i, t}(y)$ and $\sigma_{i, t}(z)$, then $l\subset B(x, 2)\subset A_i$.
\end{claim}
\begin{proof}
We have $|\frac{d\sigma_{i, t}(y)}{dt}| = |X_i(\sigma_{i, t}(y))|\leq C_1(n)$. By a direct integration, we obtain the first conclusion. The second one follows from the triangle inequality.
\end{proof}
\begin{claim}\label{cl2}
Given $t, y, z$ as in the last claim, let $l$ be the shortest geodesic connecting $\sigma_{i, t}(y)$ and $\sigma_{i, t}(z)$. Suppose $\sigma_{i, t}(y)$ is not on the cut locus of $\sigma_{i, t}(z)$, then $\frac{d(d(\sigma_{i, t}(y), \sigma_{i, t}(z)))}{dt} = \int_{l}\langle\nabla_eX_i, e\rangle ds$, where $e$ is the unit tangent vector of $l$.
\end{claim}
\begin{proof}
This is an easy consequence of the first variation of arc length.
\end{proof}
From the definition of Lie derivative, we have $$\frac{d\sigma_{i, t}^*vol(y)}{dt}=div X_i(\sigma_{i, t}(y))\sigma_{i, t}^*vol(y) = \sum\limits_k\langle\nabla_{e_k}X_i, e_k\rangle\sigma_{i, t}^*vol(y),$$ where $e_k$ is an orthonormal frame at $\sigma_{i, t}(y)$. 
\begin{claim}\label{cl3}
Define a symmetric tensor $T_{jl} = \langle\nabla_{e_j}X_i, e_l\rangle +\langle\nabla_{e_l}X_i, e_j\rangle$. Then $\int_{A_i}|T_{jl}|\leq\Phi(\frac{1}{i})$. In particular, $\int_{A_i}|div X_i|\leq \Phi(\frac{1}{i})$.
\end{claim}
\begin{proof}
As $J$ is parallel, the proof follows from (\ref{1}), (\ref{3}) and the Cauchy-Schwarz inequality.
\end{proof}
\begin{claim}\label{cl4}
Let $\Omega$ be any measurable subset of $B(x, \frac{1}{2})$, then $|\sigma_{i, t}^*vol(\Omega)-vol(\Omega)|<\Phi(\frac{1}{i})$ for $|t|\leq \frac{1}{10C_1(n)}$.
\end{claim}
\begin{proof}
By claim \ref{cl1}, $\sigma_{i, t}(B(x, \frac{1}{2}))\subset A_i$.
\begin{equation}
\begin{aligned}
&|\sigma_{i, t}^*vol(\Omega)-vol(\Omega)| = |\int_0^t\frac{d\sigma_{i, s}^*vol(\Omega)}{ds}ds|\leq \int_0^t|\frac{d\sigma_{i, s}^*vol(\Omega)}{ds}|ds \\&\leq\int_0^t\int_\Omega|div X_i(\sigma_{i, s}(y))|\sigma_{i, s}^*vol(y)ds\leq \int_0^t\int_{A_i}|div X_i(y)|dyds \leq t\Phi(\frac{1}{i}).
\end{aligned}
\end{equation}
\end{proof}
Recall the segment inequality of Cheeger-Colding \cite{[CC1]}:
\begin{prop}
Let $(Y^m, g)$ be a Riemannian manifold with $Ric\geq -(m-1)g$.  Let $A_1, A_2\subset Y^m$ be open sets such that any minimal geodesic joining $A_1, A_2$ is contained in an open set $W$. Let $D = \sup\limits_{y_1\in A_1, y_2\in A_2}d(y_1, y_2)$ and $e$ be a nonnegative function defined on $W$. Then $\int_{A_1\times A_2}\int_{0}^{\overline{y_1, y_2}}e(\gamma_{\overline{y_1, y_2}})ds \leq C(m, D)(vol(A_1)+vol(A_2))\int_W e$. Here $\gamma_{\overline{y_1, y_2}}$ is a minimal geodesic connecting $y_1, y_2$; $C(m, D)$ is a positive constant depending only on $m, D$.
\end{prop}
Now we define a function \begin{equation}\label{6}u(y, z, t) = |d(y, z) - d(\sigma_{i, t}(y), \sigma_{i, t}(z))|.\end{equation} 
\begin{prop}\label{thm0}
In (\ref{5}), $G(x, t)\leq \Phi(\frac{1}{i})$ for $|t|\leq \frac{1}{10C_1(n)}$. 
\end{prop}
\begin{proof}
By triangle inequality, $u(y, z, t)\leq 3$.
\begin{equation}
\begin{aligned}
|\frac{dF(x, t)}{dt}|&\leq6\int_{B(x, \frac{1}{2})}\int_{B(x, \frac{1}{2})}\int_{l}|\langle\nabla_eX_i, e\rangle| dsd\sigma_{i, t}^*vol(y)d\sigma_{i, t}^*vol(z)\\&+
9\int_{B(x, \frac{1}{2})}\int_{B(x, \frac{1}{2})}|div X_i(\sigma_{i, t}(y))|d\sigma_{i, t}^*vol(y)d\sigma_{i, t}^*vol(z)\\&+9\int_{B(x, \frac{1}{2})}\int_{B(x, \frac{1}{2})}|div X_i(\sigma_{i, t}(z))|d\sigma_{i, t}^*vol(y)d\sigma_{i, t}^*vol(z)\\&\leq C(n)(vol(A_i)\int_{A_i}|T_{jl}| + vol(A_i)\int_{A_i}|div X_i|)\\&\leq \Phi(\frac{1}{i}).\end{aligned}
\end{equation}
In the inequality above, we applied claim \ref{cl1}, claim \ref{cl2} and claim \ref{cl3} and proposition $1$. As $F(x, 0) = 0$, we obtain that \begin{equation}\label{7}F(x, t)\leq \Phi(\frac{1}{i}).\end{equation}
Given $a>0$, let $$E= \{(y, z)\in B(x, \frac{1}{2})\times B(x, \frac{1}{2})|u^2(y, z, t)\geq \sqrt a\}.$$ According to (\ref{4}) and (\ref{7}), when $i$ is large, $$(\sigma_{i, t}^*vol\times \sigma_{i, t}^*vol)(E)\leq \sqrt a.$$ By using the same argument as in claim \ref{cl4}, we find $$|(\sigma_{i, t}^*vol\times \sigma_{i, t}^*vol)(E)-(vol\times vol)(E)|\leq \Phi(\frac{1}{i}).$$
Therefore, \begin{equation}\label{eq-31}\begin{aligned}G(x, t)&= \int_E u^2(y, z, t) +\int_{B(x, \frac{1}{2})\times B(x, \frac{1}{2})\backslash E} u^2(y, z, t) \\&\leq 9(\Phi(\frac{1}{i})+\sqrt a)+vol(B(x, \frac{1}{2}))^2\sqrt{a}.\end{aligned}\end{equation} This completes the proof of proposition \ref{thm0}.
\end{proof}
Let $\epsilon, \delta>0$ be small numbers to be determined later. For fixed $t$
and any point $y\in B(x, \frac{1}{2})$, define $K_y = \{z\in B(x, \frac{1}{2})|u(y, z)\leq \epsilon\}$ (we have simplified $u(x, y, t)$ as $u(x, y)$). Also define $K_\delta = \{y\in B(x, \frac{1}{2})|vol(K_y)\geq vol(B(x, \frac{1}{2}))-\delta\}$.  According to proposition \ref{thm0}, \begin{equation}\label{80}vol(K_\delta)\geq vol(B(x, \frac{1}{2}))-\Phi(\frac{1}{i}).\end{equation} Thus we can find an $\epsilon$-net $x_1,....., x_N$ where $N = N(\epsilon)$ such that $x_j\in K_\delta; B(x, \frac{1}{2})\subset \cup B(x_j, \epsilon)$.  Define $K = \cap_jK_{x_j}$. For fixed $\epsilon$, first let $\delta$ be sufficiently small, then let $i$ be sufficiently large. By (\ref{80}) and the volume comparison, we may assume $vol(K)$ is so close to $vol(B(x, \frac{1}{2}))$ that $K$ is $\epsilon$-dense in $B(x, \frac{1}{2})$. 
\begin{claim}\label{cl5}
For any $y, z\in K$, $u(y, z)\leq 4\epsilon$.
\end{claim}
\begin{proof}
Since $x_1,....., x_N$ is an $\epsilon$-net, we can find $x_j$ with $d(x_j, y)\leq \epsilon$. According to the assumption, $K\subset K_{x_j}$. Thus $u(x_j, z)\leq \epsilon, u(x_j, y)\leq\epsilon$.
\begin{equation}\begin{aligned} |u(y, z)-u(x_j, z)|&\leq |d(y, z)-d(x_j, z)| + |d(\sigma_{i, t}(y), \sigma_{i, t}(z)-d(\sigma_{i, t}(x_j), \sigma_{i, t}(z))|\\&\leq d(y, x_j)+ d(\sigma_{i, t}(y), \sigma_{i, t}(x_j)) \\&\leq d(y, x_j) + u(y, x_j)+d(y, x_j)\leq 3\epsilon.\end{aligned}\end{equation} Thus $u(y, z)\leq u(x_j, z)+3\epsilon \leq 4\epsilon$.
\end{proof}

Claim \ref{5} says $\sigma_{i, t}$ is equicontinuous on $K$. By taking $\epsilon\to 0$, we find that $K$ is getting denser and denser in $B(x, \frac{1}{2})$. 
Thus we are able to take a convergent subsequence of $\sigma_{i, t}$ when $i\to\infty$. Note that the convergence is only in the measure sense. Let $\sigma_t$ be the limit of $\sigma_{i, t}$. According to the construction, $\sigma_t$ is a local isometry for small $t$.

\begin{claim}\label{cl6}
There exists a subsequence of $\sigma_{i, t}$ which converges uniformly in the measure sense for all $t$ satisfying $|t|\leq \frac{1}{10C_1(n)}$ ($C_1(n)$ appears in (\ref{2})). That is to say, given any $\epsilon>0, \eta>0$, there exists $i_0$ such that for any $i>i_0$, we can find $K_i\subset A_i$ so that for any $x, y\in K_i$, $u(x, y, t)\leq\epsilon$ and $vol(K_i)\geq vol(A_i)-\eta$ for $|t|\leq\frac{1}{10C_1(n)}$.\end{claim}
\begin{proof}
Divide $[-\frac{1}{10C_1(n)}, \frac{1}{10C_1(n)}]$ into finitely many subintervals so that the length of each subinterval is small.  Then we can prove that the convergence for the endpoints of each subinterval. Note that the generating vector field $X_i$ has bounded length. Then the proof follows from the triangle inequality.
\end{proof}

\begin{cor}\label{cor2}
For any $t_1, t_2$ with $|t_1|, |t_2|\leq \frac{1}{10C_1(n)}$, $\sigma_{t_1}\circ\sigma_{t_2} = \sigma_{t_2}\circ\sigma_{t_1}$. \end{cor}
\begin{proof}
As $\sigma_{i, t}$ is generated by $X_i$, $\sigma_{i, t_1}\circ\sigma_{i, t_2} = \sigma_{i, t_2}\circ\sigma_{i, t_1}$.
The proof follows from claim \ref{cl4} and claim \ref{cl6}.\end{proof}

Let $r(x)=d(q, x)$ on $X_q$. As $r^{-1}[5, 6]$ is compact in $B(q, 7)\backslash B(q, 4)$, we can cover it by finitely many small balls $B_j$ with radius $\frac{1}{2}$. Then on each $B_j$ we have isometry $\sigma^j_t$. By taking further subsequences, we may assume $\sigma^j_t$ coincides on the overlap. Then we glue $\sigma^j_t$ together. Say $\sigma_t$ is defined on $r^{-1}[5, 6]$. By (\ref{2}) and the fact that $\sigma_{i, t}$ preserves $\rho_i$, we obtain that $r^{-1}[5, 6]$ is invariant under $\sigma_t$. 
Next we extend the map $\sigma_t$ for all $t\in\mathbb{R}$. Namely, if $|t|>\frac{1}{10C_1(n)}$, we can write $t = \sum_j t_j$ where $|t_j|\leq \frac{1}{10C_1(n)}$ for each $j$.
Define $\sigma_t = \sigma_{t_1}\circ\cdot\cdot\cdot\circ\sigma_{t_j}$. By corollary \ref{cor2}, $\sigma_t$ is well defined. 
\begin{claim}\label{cl7}
In $r^{-1}[5, 6]$, $\sigma_t$ commutes with the homothety map $\lambda_s: (u, r)\to (u, e^sr)$. Here $u$ is in the cross section $\Sigma$, $s\in\mathbb{R}$. Moreover, $\sigma_t$ is an isometry on the cross section of the metric cone. 
\end{claim}
\begin{proof}
Obviously $\sigma_t$ is an local isometry in $r^{-1}[5, 6]$. By the relation between the distance function on $\Sigma$ and $X_q$, we obtain that $\sigma_t$ is a local isometry on the cross section $\Sigma$ for fixed $r$.  A simple argument implies that the distance function on $\Sigma$ is nonincreasing under the map $\sigma_t$. 
As $\sigma_t$ is a homeomorphism on $\Sigma$ (it has an inverse $\sigma_{-t}$), $\sigma_t$ must be an isometry on $\Sigma$ for each $r$.
Next we prove that $\sigma_t$ commutes with the homothety map. It suffices to prove this for small $t$ and $s$. For $u\in\Sigma$, let $\sigma_t(u, r) = (u_{t, r}, r)$. We only need to prove that $u_{t, r}$ is independent of $r$.
 As $t$ and $s$ are small, $d((u, r), (u, e^sr))$ is small. Hence, by the local isometry of $\sigma_t$, $|e^s-1|r=d((u, r), (u, e^sr)) = d((u_{t, r}, r), (u_{t, e^sr}, e^sr))\geq |e^s-1|r$.
By the distance formula for the warped product metric, $u_{t, r} = u_{t, e^{s}r}$. This proves the commutativity.

\end{proof}

By claim \ref{cl7} and the distance formula for warped product metric, we easily extend $\sigma_t$ as an isometry on $X_q$ which commutes with the homothety map.

\medskip

Next we prove that $\sigma_t$ is locally free on $X_q\backslash q$. As $\sigma_t$ is isomorphic to $\mathbb{R}$, it suffices to prove that $\sigma_t$ has no common fixed point except the vertex $q$. We argue by contradiction. Without loss of generality, assume $\sigma_t(x) = x$ for all $t\in\mathbb{R}$, where $5< r(x) <6$. Consider a tangent cone at $x$, say $\mathbb{R}^k\times W$. Here $W$ is a metric cone without Euclidean factor. Let us assume the vetex of $\mathbb{R}^k\times W$ is $(0^k, w^*)$. As $X_q$ is a metric cone, $k\geq 1$. In fact, by theorem $9.1$ in \cite{[CCT]}, $k$ must be even.

 Let $\epsilon>0, \delta>0$ be very small fixed constants and $N$ be a very large constant so that $N\epsilon$ is still very small. Say take $N = \frac{1}{\sqrt{\epsilon}}$.  We may assume $$d_{GH}(B(x, N\epsilon), B_{\mathbb{R}^{k}\times W}((0^k, w^*), N\epsilon))\leq \frac{\delta\epsilon}{20}.$$ Recall $(N_i, q_i)\to (X_q, q)$ which is considered in the beginning of this section. Take $x_i\in N_i$ with $x_i\to x$. Then, for all large $i$,  $$d_{GH}(B(x_i, N\epsilon), B_{\mathbb{R}^{k}\times W}((0^k, w^*), N\epsilon))\leq \frac{\delta\epsilon}{10}.$$ Note that $5 \leq r_i(x_i):=d(x_i, q_i)\leq 6$.

 Let $F_i$ be the scale map $B(x_i, N\epsilon, g_i)\to B(y_i, N, \frac{1}{\epsilon^2}g_i)$.  For notational convenience, we simplify $B(y_i, N, \frac{1}{\epsilon^2}g_i)$ as $B(y_i, N)$. Then 
 \begin{equation}\label{14}d_{GH}(B(y_i, N), B_{\mathbb{R}^{k}\times W}((0^k, w^*), N))\leq \delta.\end{equation} Recall $X_i=J\nabla\rho_i$. Define $$Y_i = \frac{\epsilon}{r_i(x_i)}(F_i)_*X_i$$ in $B(y_i, 10)$. By (\ref{2}), \begin{equation}\label{eq7}|Y_i|\leq C_1(n).\end{equation} Note $\epsilon$ is independent of $i$.
Then by (\ref{1}) and the volume comparison,  
\begin{equation}\label{eq101}\dashint_{B(x_i, 10\epsilon)}|\nabla\rho_i-\nabla \frac{1}{2}r_i^2|^2 + |\nabla^2\rho_i-g_i|^2<\Phi(\frac{1}{i}|\epsilon).\end{equation} Thus $$\dashint_{B(x_i, 10\epsilon)}||X_i|-r_i(x_i)|\leq 20\epsilon, \dashint_{B(x_i, 10\epsilon)}|\nabla X_i|^2\leq 4n.$$ Therefore
\begin{equation}\label{9}\dashint_{B(y_i, 10)}||Y_i|-1|\leq 20\epsilon, \dashint_{B(y_i, 10)}|\nabla Y_i|^2\leq 4n\epsilon^2.\end{equation}

We need two lemmas in \cite{[CCT]}:
\begin{lemma}\label{lm1}[Lemma 9.5]
Let $V$ be an open subset of a complete manifold $M$. Let $\lambda>0$ be the smallest nonzero eigenvalue of Laplacian on $V$ with Dirichlet boundary conditions. Let $h: V\to \mathbb{R}$ be Lipschitz and let $v$ be a vector field on $V$ such that
\begin{equation}\label{10}
\sup\limits_{V} |h|\leq c; \dashint_V|\nabla h-v|^2\leq\tilde\delta^2; \dashint_V|div v|\leq\tilde\delta.
\end{equation}
Then if $b$ denotes the harmonic function such that $b|_{\partial V} = h|_{\partial V}$,
\begin{equation}\label{11}
\dashint_V|\nabla h-\nabla b|^2\leq (4c+\tilde\delta)\tilde\delta; \dashint_V(h-b)^2\leq (4c+\tilde\delta)\tilde\delta \lambda^{-1}.
\end{equation}
\end{lemma}
\begin{lemma}\label{lm2}[Lemma 9.14]
Let $M^n$ be a Riemannian manifold with $Ric_{M^n}\geq -(n-1)\tilde\epsilon^2$, $m\in M$. Let $v$ be a vector field on $M$ so that
\begin{equation}\label{12}
\sup\limits_{B(m, 1)}|v|\leq c; \dashint_{B(m, 1)}|\nabla v|^2\leq\tilde\delta^2; (1-\tilde\delta)\frac{vol(B(m, 1))}{vol(B(\underline{m}, 1))}\leq \frac{vol(B(m, 2))}{vol(B(\underline{m}, 2))}.\end{equation}
Here $\underline{m}$ is a point in the complete simply connected space form $\underline{M}^n$ with constant curvature $-\tilde\epsilon^2$. Then there exists $h: B(m, 1)\to\mathbb{R}$ so that 
\begin{equation}\label{13}
\sup\limits_{B(m, 1)}|h|\leq c(n, c); \dashint_{B(m, 1)}|\nabla h-v|^2\leq\Phi(\tilde\epsilon, \tilde\delta|n, c).
\end{equation}
\end{lemma}

We apply lemma \ref{lm2} to $B(y_i, 10)$ and the vector field $Y_i$. That is, take $m = y_i$ and $v = Y_i$. Note the radii are different, but this does not affect the proof of the lemma.
(\ref{14}) and the volume convergence theorem of Colding \cite{[C]} imply the last condition in (\ref{12}). (\ref{eq7}) and (\ref{9}) imply the first two conditions of (\ref{12}).
Then we apply lemma \ref{lm1}: take $V = B(y_i, 10), v = Y_i$ while $h$ is given by lemma \ref{lm2}. Note that in this case, $\lambda$ has a positive lower bound which depends only on $n$.

(\ref{13}) implies that $|h|$ is bounded by $c(n)$. By the construction of $b$ in lemma \ref{lm1}, $|b|$ is also bounded by $c(n)$. Then Cheng-Yau gradient estimate \cite{[CY]} implies 
\begin{equation}\label{16}|\nabla b|\leq c_1(n)\end{equation} in $B(y_i, 5)$. (\ref{10}) and (\ref{11}) imply 
\begin{equation}\label{17}
\dashint_{B(y_i, 10)}|Y_i-\nabla b|^2\leq\Phi(\epsilon, \delta, \frac{1}{i}).
\end{equation}
Note that the right hand side is not going to $0$ as $i\to\infty$, since $\epsilon, \delta$ are fixed.
Then (\ref{17}) and (\ref{9}) imply 
\begin{equation}\label{18}
\dashint_{B(y_i, 10)}|\langle\nabla b, Y_i\rangle-1|\leq\Phi(\epsilon, \delta, \frac{1}{i}).
\end{equation}
Let $\overline\sigma_{i, t}$ be the flow generated by $Y_i$ in $B(y_i, 5)$ for $|t|\leq \frac{1}{10C_1(n)}$. Let $1>\tau>0$ be a small number, to be determined later. Define $$U_i = \{z\in B(y_i, \tau)|\overline\sigma_{i, \frac{1}{10C_1(n)}}(z)\in B(y_i, 2\tau)\}.$$Recall $y_i=F_i(x_i)$ and $x_i\to x$. According to claim \ref{cl6} and the assumption that $\sigma_t(x) = x$, 
 \begin{equation}\label{8}
 vol(U_i)\geq vol(B(y_i, \tau))-\Phi(\frac{1}{i}|\tau).
 \end{equation}

Define $$h(i, t) = \int_{z\in B(y_i, \tau)}|b(z) +t-b(\overline\sigma_{i, t}(z))|^2d\overline\sigma_{i, t}^*vol(z),$$ for $|t|\leq \frac{1}{10C_1(n)}$.
One the one hand, \begin{equation}\label{15}\begin{aligned}
|\frac{dh}{dt}|&\leq C\int_{z\in B(y_i, \tau)}|\langle\nabla b, Y_i\rangle_{\overline\sigma_{i, t}(z)}-1|d\overline\sigma_{i, t}^*vol(z)+C\int_{z\in B(y_i, \tau)}|\frac{d\overline\sigma_{i, t}^*vol(z)}{dt}|\\&\leq \Phi(\epsilon, \frac{1}{i}, \delta).\end{aligned}
\end{equation}
 In (\ref{15}), we have used (\ref{18}) and similar arguments as in claim \ref{cl4}.
Note $h(i, 0) = 0$. This proves 
\begin{equation}\label{19}|h(i, \frac{1}{10C_1(n)})|<\Phi(\epsilon, \frac{1}{i}, \delta).\end{equation}

On the other hand, if $\tau\leq\frac{1}{400c_1(n)C_1(n)}$, (\ref{16}) implies that $|b(z_1)-b(z_2)|\leq \frac{1}{100C_1(n)}$ for $z_1, z_2\in B(y_i, 2\tau)$.
Therefore
\begin{equation}\label{20}
\begin{aligned}
h(i, \frac{1}{10C_1(n)})&\geq \int_{U_i}|b(z) +\frac{1}{10C_1(n)}-b(\overline\sigma_{i, \frac{1}{10C_1(n)}}(z))|^2d\overline\sigma_{i, \frac{1}{10C_1(n)}}^*vol(z)\\&\geq
\frac{1}{20C_1(n)}\int_{U_i}d\overline\sigma_{i, \frac{1}{10C_1(n)}}^*vol(z)\\&\geq \frac{1}{20C_1(n)}(vol(B(y_i, \tau))-\Phi(\frac{1}{i}, \epsilon|\tau)).\end{aligned}
\end{equation}
Here, we have used (\ref{8}) and similar arguments as in claim \ref{cl4}. Take $$\tau = \min(1, \frac{1}{400c_1(n)C_1(n)}).$$ If $\epsilon, \delta$ are sufficiently small, (\ref{20}) contradicts (\ref{19}).

The proof of theorem \ref{thm1} is complete.

\end{proof}

\section{\bf{Applications to complete K\"ahler manifolds with nonnegative bisectional curvature}}

In this section we study the limit of analytic functions on Gromov-Hausdorff limit of K\"ahler manifolds with nonnegative bisectional curvature.
This topic has recently been studied in papers \cite{[L1]}-\cite{[L5]}. First recall some definitions and results in \cite{[L1]}.

On a K\"ahler manifold $M^n$, a holomorphic function $f\in\mathcal{O}_d(M)$ if $|f(x)|\leq c(r(x)+1)^d$. Here $r(x)$ is the distance to a fixed point, $c$ is independent of $r$.
\begin{definition}
Let $M^n$ be a complete K\"ahler manifold and $f\in \mathcal{O}(M)$. The order at infinity is defined by $\overline{\lim\limits_{r\to\infty}}\frac{\log M_f(r)}{\log r}$, where $r$ is the distance to a fixed point $p$ on $M$, $M_f(r)$ is the maximal modulus of $f$ on $B(p, r)$.
\end{definition}
\begin{remark}
Note that Colding and Minicozzi made a similar definition for harmonic functions over a metric cone (Definition $1.32$, \cite{[CM]}).
\end{remark}
\begin{prop}[Corollary $3$, \cite{[L1]}]
Let $M^n$ be a complete K\"ahler manifold with nonnegative holomorphic sectional curvature. If $f\in\mathcal{O}_{d+\epsilon}(M)$ for any $\epsilon>0$. Then $f\in\mathcal{O}_d(M)$.
\end{prop}
Let $(M, g)$ be a complete K\"ahler manifold with nonnegative bisectional curvature.  Assume $f$ is a nonconstant holomorphic function of polynomial growth on $M$. Let $d<\infty$ is the order of $f$ at infinity (note $d$ might not be an integer). According to the proposition above, $f\in\mathcal{O}_d(M)$.

 We further assume $M^n$ is of maximal volume growth. Note this holds automatically when the universal cover does not split as a product. See theorem $2$ in \cite{[L2]}. Fix a point $p\in M$. Consider a tangent cone at infinity $(M_\infty, g_\infty, p_\infty)$ which is the Gromov-Hausdorff limit of $(M_i, g_i, p_i) = (M, r_i^{-2}g,  p)(r_i\to\infty)$. 
Define a rescaled function 
\begin{equation}\label{21}
f_i(x) = \frac{f(x)}{M_f(r_i)}.
\end{equation} Then $\sup |f_i(x)| = 1$ for $x\in B(p, r_i)$.  
\begin{prop}[Theorem $2$, \cite{[L1]}]
Let $M$ be a complete K\"ahler manifold with nonnegative holomorphic sectional curvature. Then $\log M_f(r)$ is convex in terms of $\log r$. Therefore, $\frac{M_f(kr)}{M_f(r)}$ is monotonic increasing for $k>1$. If $f\in \mathcal{O}_d(M)$,  then $\frac{M_f(r)}{r^d}$ is nonincreasing.
\end{prop}
By this proposition, 
\begin{equation}\label{24}|f_i(x)|\leq (2k)^d\end{equation} on $B(p, 2kr_i)$ for $k>1$. By Cheng-Yau's gradient estimate \cite{[CY]},  \begin{equation}\label{25}|\nabla_{g_i}f_i(x)|\leq a(n)k^{d-1}\end{equation} on $B(p, kr_i)$. Thus there exists a subsequence of $f_i$ converging to $f_\infty$ uniformly on each compact set on $M_\infty$. By theorem \ref{cc}, $M_\infty$ is a metric cone, say $M_\infty = \Sigma\times_{r^2}\mathbb{R}^+$ with vertex $p_\infty$. Denote points on $M_\infty$ by $(u, r)$ where $u\in\Sigma$, $r\geq 0$. 
\begin{prop}\label{prop-100}
$f_\infty$ is a homogeneous function of degree $d$. That is, \begin{equation}
\label{23}f_\infty(u, e^sr) = e^{ds}f_\infty(u, r).\end{equation}
\end{prop}
\begin{proof}
By the proposition above,  for any $k>1$, $\frac{M_f(kr)}{M_f(r)}$ is monotonic increasing. Since $d = \overline{\lim\limits_{r\to\infty}}\frac{\log M_f(r)}{\log r}$,  $$\lim\limits_{r\to\infty}\frac{M_f(kr)}{M_f(r)} =k^d.$$ Let $M_{f_\infty}(r)$ be the maximal modulus of $f_\infty$ on $B(p_\infty, r)$. Then for any $r_1>r_2>0$,
$$\frac{M_{f_\infty}(r_1)}{M_{f_\infty}(r_2)} = \lim\limits_{i\to\infty}\frac{M_f(r_1r_i)}{M_f(r_2r_i)} = \frac{r_1^d}{r_2^d}.$$ Since $M_{f_\infty}(1) = 1$,
\begin{equation}\label{eq-106}
M_{f_\infty}(r) = r^d.
\end{equation}

$f_i$ are harmonic functions. By Lemma \ref{lemma-10}, $f_\infty$ is harmonic on $C(\Sigma)$. Also it is of polynomial growth. 
It is easy to see that the $(2n-1)$ dimensional Hausdorff measure on $\Sigma$ satisfies the volume doubling property and the weak Poincare inequality. See lemma $4.3$ in \cite{[Di1]} for a proof. Also, one can directly check that $\Sigma$ is $\nu$-rectifiable. Therefore, we have a Laplacian operator on $\Sigma$.

On the metric cone $C(\Sigma)$, there is a decomposition formula (see \cite{[Di2]}\cite{[Di1]}).
\begin{equation}\label{eq-99}
\Delta u = -\frac{\partial^2 u}{\partial r^2} -\frac{2n-1}{r}\frac{\partial u}{\partial r}+\frac{1}{r^2}\Delta_\Sigma u.
\end{equation}
Therefore, if $\phi_i$ is the $i$-th eigenfunction of $\Delta_\Sigma$ with eigenvalue $\lambda_i$, then $r^{\alpha_i}\phi_i(x)$ is harmonic. Here 
\begin{equation}\label{eq-70}
\lambda_i=\alpha_i(2n+\alpha_i-2).
\end{equation} We normalize so that $\dashint_\Sigma|\phi_i|^2 = 1$. For any harmonic function (complex function) $u$ on $X$,  we can write (see \cite{[Che0]}\cite{[Di2]})
\begin{equation}\label{eq-105}
u = \sum\limits_{i=0}^\infty c_ir^{\alpha_i}\phi_i.
\end{equation}
Here $c_i$ are complex constants.
Define $I(r) = \frac{1}{Vol(\partial B(p_\infty, r))}\int_{\partial B(p_\infty, r)}|u|^2$. Then $$I(r) = \sum\limits_{i=0}^\infty |c_i|^2r^{2\alpha_i}.$$
This implies that if $u$ is of polynomial growth on $C(\Sigma)$, there are only finitely many terms in (\ref{eq-105}).
\begin{claim}\label{cl-5}
$f_\infty = r^d\phi$ for some $\phi$ on $\Sigma$ with $\Delta_\Sigma\phi = d(2n+d-2)\phi$.
\end{claim}
\begin{proof}
Since $f_\infty$ is harmonic, by (\ref{eq-105}) and (\ref{eq-106}), if $r\to 0$, we find $\alpha_i\geq d$ for all $i$. If $r\to \infty$, we find $\alpha_i\leq d$ for all $i$.
\end{proof}
The proof of the proposition is complete.
\end{proof}

According to theorem \ref{thm1}, there exist isometries $\sigma_t$ on $M_\infty$. The next result states the behavior of $f_\infty$ under $\sigma_t$. 
\begin{prop}\label{thm-10}
Let $M^n, M_i^n, M_\infty$ be stated as above.  Then $f_\infty(\sigma_t(x)) = e^{\sqrt{-1}dt}f_\infty(x)$ for any $x\in M_\infty$ and $t\in\mathbb{R}$. 
\end{prop}
\begin{proof}
We may assume $x\neq p_\infty$, otherwise the conclusion is obvious, as $f(p_\infty) = 0$. As regular points are dense on $M_\infty$ and $\sigma_t$ preserves regular points (isometry), we can also assume $x$ is a regular point. 
It suffices to prove that for any sequence $t_j\to 0$, there exists a subsequence so that 
\begin{equation}\label{22}\lim\limits_{j\to\infty}\frac{f_\infty(\sigma_{t_j}(x))-f_\infty(x)}{t_j} = \sqrt{-1}df_\infty(x).\end{equation} We further assume $r(x) = 1$ (the general case follows from a rescaling). Define \begin{equation}\label{eq-119}(N'_j, g'_j, x'_j)=(M_\infty, t_j^{-2}g_\infty, x).\end{equation} 
Then $(N'_j, g'_j, x'_j)$ converges to $(\mathbb{R}^{2n}, 0)$ in the pointed Gromov-Hausdorff sense. 
Recall $(M_i, g_i, p_i)$ pointed converges to $(M_\infty, g_\infty, p_\infty)$. Take a function $\mathbb{N}\to\mathbb{N}$: $i = i(j)$ which is increasing sufficiently fast.
Consider points $x''_i\in M_i$ converging to $x\in M_\infty$, define \begin{equation}\label{eq-120}(N_j, g_j, x_j) = (M_{i(j)}, t_j^{-2}g_{i(j)}, x''_{i(j)}).\end{equation}
We may assume \begin{equation}\label{eq-121}d_{GH}(B_{N_j}(x_j, K_j), B_{N'_j}(x'_j,  K_j))\leq\frac{1}{K_j},\end{equation} where $K_j$ is a sequence going to infinity.
Thus $(N_j, g_j, x_j)$ pointed converges to ($\mathbb{R}^{2n}, 0$).  Note $f_j := f_{i(j)}$ are holomorphic functions on $N_j$. 

According to \cite{[CC1]}, for any $R>0$, there exist harmonic functions $b_1,..., b_{2n}$ on $B(x_j, 3R)$ such that \begin{equation}\label{eq4}\dashint_{B(x_j, 2R)}  \sum\limits_{k=1}^{2n}|\nabla(\nabla b_k)|^2 +\sum\limits_{k, l}|\langle\nabla b_k, \nabla b_l\rangle - \delta_{kl}|^2\leq \Phi(\frac{1}{i}|R, n);\end{equation} \begin{equation}\label{eq5}|\nabla b_k|\leq C_1(n)\end{equation} on $B(x_j, 2R)$.  
According to equation (9.25) and the first paragraph of page $912$ in \cite{[CCT]}, we may also assume \begin{equation}\label{eq14}\dashint_{B(x_j, R)} |J\nabla b_{2s-1} - \nabla b_{2s}|^2 \leq \Phi(\frac{1}{i}|n, R)\end{equation} for $1\leq s\leq n$.
By taking $R\to\infty$ and a diagonal subsequence argument, we can define a linear complex structure on the limit space $\mathbb{R}^{2n}$: \begin{equation}\label{eq15}J\nabla b_{2s-1} = \nabla b_{2s}, J\nabla b_{2s} = -\nabla b_{2s-1}\end{equation} for 
$1\leq s\leq n$. In this way, we identify $\mathbb{R}^{2n}$ with $\mathbb{C}^n$.
On $N_j$, define \begin{equation}\label{eq-118}h_j(y) = \frac{f_j(y)-f_j(x_j)}{t_j}.\end{equation} By the local bound of $f_j$ and the gradient estimate (\ref{24}), (\ref{25}), $h_j$ are local Lipchitz functions on $N_j$: \begin{equation}\label{26}|\nabla h_j|\leq c(n), h_j(x_j) = 0\end{equation} in $B(x_j, \frac{1}{10t_j})$.   Arzela-Ascoli theorem implies that a subsequence of $h_j$ converges uniformly in each compact set to a Lipchitz function $h$ on $\mathbb{C}^n$.
\begin{claim}\label{cl-100}
$h$ is complex linear on $\mathbb{C}^n$ with $h(0) = 0$.
\end{claim}
\begin{proof}
By (\ref{eq-118}) and that $(N_j, g_j, x_j)$ pointed converges to $(\mathbb{R}^{2n}, 0) = (\mathbb{C}^n, 0)$, $h(0) = 0$.
As $h$ is Lipchitz, it suffices to prove $h$ is holomorphic on $\mathbb{C}^n$. The proof of this fact is contained in lemma $4$ in \cite{[L2]}.
\end{proof}

Recall the function $\rho_i$ appeared in (\ref{1}).
To continue the proof, we consider the flow $\lambda_i(t)$ on $M_i$ generated by $\nabla\rho_i$.  For simplicity, we only consider the flow in $B_i=B(p_i, 2)\backslash B(p_i, \frac{1}{2})$ and $|t|\leq \frac{1}{10C_1(n)}$ ($C_1(n)$ appears in (\ref{2})). 

\begin{claim}\label{cl-20}
Let $\Omega$ be any measurable subset of $B_i$, then $|\lambda_{ i}(t)^*vol(\Omega)-e^{2nt}vol(\Omega)|<\Phi(\frac{1}{i})$ for $|t|\leq \frac{1}{10C_1(n)}$.
\end{claim}
\begin{proof}
By (\ref{1}) and similar arguments as in lemma \ref{cl4}, we see $|\frac{d\lambda_{ i}(t)^*vol(\Omega)}{dt}- 2n\lambda_{ i}(t)^*vol(\Omega)|<\Phi(\frac{1}{i})$. Then the proof follows from integral estimate.

\end{proof}
\begin{lemma}\label{lm-20}
For $y\in B_i$ and $|t|\leq \frac{1}{10C_1(n)}$, $\lim\limits_{i\to\infty}(r_i(\lambda_i(t, y)) -e^t r_i(y)) = 0$ in the measure sense. Here $r_i(y) = d(p_i, y)$.
\end{lemma}
\begin{proof}
Define $$H_i(t) = \int_{B_i}|\rho_i(\lambda_i(t, y))-e^{2t}\rho_i(y)|^2dy.$$ By (\ref{2}), $\rho_i- \frac{1}{2}r_i^2\to 0$ as $i\to \infty$. It suffices to prove that $$\lim\limits_{i\to\infty}H_i(t) = 0.$$ As $\rho_i(\lambda_i(t, y))$ is bounded for $|t|\leq\frac{1}{10}$ and $y\in B_i$, there exists a constant $C$ so that\begin{equation}\label{eq41}\begin{aligned}|\frac{dH_i}{dt}| &\leq C \int_{B_i}||\nabla\rho_i(\lambda_i(t, y))|^2-2e^{2t}\rho_i(y)|dy\\&\leq C2H_i^\frac{1}{2}(vol(B_i))^\frac{1}{2}+C\int_{B_i}||\nabla\rho_i(\lambda_i(t, y))|^2-2\rho_i(\lambda_i(t, y))|dy.\end{aligned}\end{equation} 
With (\ref{1}) and (\ref{2}), one can easily show that $$\int_{B_i}||\nabla\rho_i(\lambda_i(t, y))|^2-2\rho_i(\lambda_i(t, y))|d\lambda_{ i}(t)^*vol(y)< \Phi(\frac{1}{i}).$$ 
Then by claim \ref{cl-20} and similar arguments between (\ref{7}) and (\ref{eq-31}), $$\int_{B_i}||\nabla\rho_i(\lambda_i(t, y))|^2-2\rho_i(\lambda_i(t, y))|dy< \Phi(\frac{1}{i}).$$ Note this is the last term of (\ref{eq41}). ODE arguments give the proof.\end{proof}

\begin{claim}\label{cl-19}
For $y\in B_i$ and $|t|\leq \frac{1}{10C_1(n)}$, $\lim\limits_{i\to\infty}(d(\lambda_i(t, y), y) - |e^t-1|r_i(y))\leq 0$  in the measure sense.
\end{claim}
\begin{proof}
We use similar arguments as in lemma $2.1$, step $2$ in \cite{[KW]}.
\begin{equation}\label{eq42}
\begin{aligned}
\int_{B_i}\int_0^t||\nabla\rho_i(\lambda_i(s, y))|-e^s\sqrt{2\rho_i(y)}|dsdy&=\int_0^t\int_{B_i}||\nabla\rho_i(\lambda_i(s, y))|-e^s\sqrt{2\rho_i(y)}|dyds\\&\leq 4\int_0^t\int_{B_i}||\nabla\rho_i(\lambda_i(s, y))|^2-2e^{2s}\rho_i(y)|dyds\\&<\Phi(\frac{1}{i}).\end{aligned}
\end{equation}
In the equation, we used that $e^t\sqrt{2\rho_i(y)}\to e^tr_i(y)> \frac{1}{4}$ and (\ref{eq41}).
Note that the length of the curve $\lambda_i(s, y)$ for $s\in[0, t]$ ($[t, 0]$ if $t\leq 0$) is $\int_0^t |\nabla\rho_i(\lambda_i(s, y))|ds$. 
Thus the proof of the claim easily follows by integration.
\end{proof}
\begin{prop}\label{prop-10}
As $i\to\infty$, $\lambda_i(t)$ converges to $\lambda(t)$ in the measure sense. Here $\lambda(t)(u, r) = (u, e^tr)$, the gradient flow of $\frac{r^2}{2}$ on $M_\infty$.
\end{prop}
\begin{proof}
In view of the distance formula on $M_\infty$, the proposition is a consequence of lemma \ref{lm-20} and claim \ref{cl-19}.
\end{proof}

Let $F_j$ be the rescale map from $(M_{i(j)}, g_{i(j)}, x''_{i(j)})$ to $(N_j, g_j, x_j)$. On $N_j$, define vector fields $$V_j=t_j(F_j)_*\nabla\rho_{i(j)} , W_j=t_j(F_j)_*J\nabla\rho_{i(j)}.$$ Given any $R>0$, if $i(j)$ is increasing sufficiently fast, by similar arguments as in (\ref{eq101}), we find 
\begin{equation}\label{eq102}\dashint_{B(x''_{i(j)}, Rt_j)}|\nabla\rho_{i(j)}-\nabla \frac{1}{2}r_{i(j)}^2|^2 + |\nabla^2\rho_{i(j)}-g_{i(j)}|^2<\Phi(\frac{1}{j}|R).\end{equation}
Recall $r(x) = 1$, hence on $M_{i(j)}$, $r_{i(j)}(x''_{i(j)})\to 1$. Thus similar as (\ref{9}),
\begin{equation}\label{eq90}\dashint_{B(x_j, R)}||V_j|-1|+|\nabla V_j|^2+||W_j|-1|+|\nabla W_j|^2\leq \Phi(\frac{1}{j}|R, n).\end{equation}
Note $JV_j = W_j$ on $N_j$. By similar arguments as in (\ref{17}) and (\ref{eq14}), 
we obtain functions $b^j_k(k=1,.., 2n)$ on $B(x_j, 2R)$ satisfying (\ref{eq14}) and
\begin{equation}\label{eq1-3}\int_{B(x_j, R)}|V_j-\nabla b^j_1|^2+|W_j-\nabla b^j_2|^2+\sum\limits_{k=1}^{2n}(|\nabla^2 b^j_k|^2+\sum\limits_{1\leq k, l\leq 2n}|\langle\nabla b^j_k, \nabla b^j_l\rangle -\delta_{kl}|^2<\Phi(\frac{1}{j}|R). \end{equation} As $R\to\infty$, these functions $b^j_k$ converge in a subsequence to standard real coordinates $b_k$ on $\mathbb{C}^n$ which is identified with a tangent cone at $x$. Below we shall use these coordinates.
Let $\overline\sigma_j(t)$ and $\overline\lambda_j(t)$ be the flows on $N_j$ generated by $W_j$ and $V_j$.
\begin{lemma}\label{lm-100}
In the measure sense, $\lim\limits_{j\to\infty}(b^j_k(\overline\sigma_j(t)(p))- b^j_k(p)-t\delta_{k2})=0$ for $p\in B(x_j, 10)$. Similarly, $\lim\limits_{j\to\infty}(b^j_k(\overline\lambda_j(t)(p)) - b^j_k(p)-t\delta_{k1})=0.$ 
\end{lemma}
\begin{proof}
Define $K(k, t) = \int_{z\in B(x_j, 20)}|b^j_k(z) +\delta_{k2}t-b^j_k(\overline\sigma_{j}(t)(z))|^2d(\overline\sigma_{j}(t))^*vol(z).$ Then we can prove $|\frac{dK}{dt}|\to 0$ as in (\ref{15}). The proof is done by simple integration.
\end{proof}

By (\ref{eq-119}), (\ref{eq-120}) and (\ref{eq-121}), let $q_j$ be a pre-image of $\sigma_{t_j}(x)$ in $N_j$, given by the Gromov-Haudorff approximation. In view of claim \ref{cl6} and lemma \ref{lm-100}, we find \begin{equation}\label{eq-117}q_j\to(0, 1, 0,...., 0)\in\mathbb{C}^n.\end{equation} Note there is no issue to apply the measured convergence above.

Obviously $t_j$ is fixed for each $j$. Letting $i=i(j)$ increase sufficiently fast,  by (\ref{eq-120}) and (\ref{eq-118}), we find \begin{equation}\label{eq-116}|h_j(q_j)-\frac{f_\infty(\sigma_{t_j}(x))-f_\infty(x)}{t_j}|\to 0.\end{equation}

Similarly, let $q'_j$ be the pre-image of $\lambda_{t_j}(x)$ in $N_j$. Then proposition \ref{prop-10} and lemma \ref{lm-100} imply \begin{equation}\label{eq-115}q'_j\to (1, 0, 0, ..., 0)\in\mathbb{C}^n\end{equation} Similar as in (\ref{eq-116}), \begin{equation}\label{eq-114}|h_j(q'_j)-\frac{f_\infty(\lambda_{t_j}(x))-f_\infty(x)}{t_j}|\to 0.\end{equation} By claim \ref{cl-100}, (\ref{eq-117}), (\ref{eq-115}) and that $h_j\to h$ uniformly on each compact set, \begin{equation}\label{eq-113}\lim\limits_{j\to\infty}h_j(q_j)=h(0, 1, .... 0) = \sqrt{-1}h(1, 0, ....., 0)=\sqrt{-1}\lim\limits_{j\to\infty}h_j(q'_j).\end{equation} According to (\ref{23}), \begin{equation}\label{eq-112}\lim\limits_{j\to\infty}\frac{f_\infty(\lambda_{t_j}(x))-f_\infty(x)}{t_j} = df_\infty(x).\end{equation} Putting (\ref{eq-116}), (\ref{eq-114}), (\ref{eq-113}) and (\ref{eq-112}) together, we proved (\ref{22}).
This completes the proof of proposition \ref{thm-10}.\end{proof}

\begin{cor}\label{cor1}
Let $M^n$ be a complete K\"ahler manifold with nonnegative bisectional curvature and maximal volume growth. Let $X=\Sigma\times_{r^2}\mathbb{R}^+$ be a tangent cone at infinity. Let $E = \{q|$q is the order at infinity for some $f\in\mathcal{O}_d(M), d\geq 0\}$ and $v = dim_{\mathbb{Q}}(E\otimes_\mathbb{Z}\mathbb{Q})$. Let $\mathbb{T}^u$(dimension $u$) be the torus given by the closure of $\sigma_t$ in the isometry group of $\Sigma$. Then $u\geq v$. In particular, if $u=1$, then orders of holomorphic functions with polynomial growth are rationally related.
\end{cor}
\begin{remark}
In the standard $\mathbb{C}^n$ case, $u=v=1$ (Hopf fibration).
Consider the case $\mathbb{C}\times N$, where $N$ is conformal to $\mathbb{C}$, with rotationally symmetric metric of nonnegative curvature. We can also adjust the metric so that $\pi$ is the order of some $f\in\mathcal{O}_d(N)$ at infinity. In this case $u = 2 = v$.
\end{remark}

\begin{proof}
Let $f^i(1\leq i\leq m)$ be holomorphic functions of polynomial growth on $M$ (here $m\in\mathbb{N}$ is arbitrary), with orders $d_i$ at infinity. We assume the limit of these functions on $X$ are $f_\infty^i$. According to (\ref{23}), $f^i_\infty$ is homogeneous of degree $d_i$. Note the limit functions are not necessarily unique. This does not affect our argument. 

Pick $p\in X$ so that $f^i_\infty(p)\neq 0$ for any $i$.  For $t\in\mathbb{T}^u$, define functions $g_i(t) = f^i_\infty(t(p))$.
Write $\mathbb{T}^u = (e^{2\pi\sqrt{-1}x_1},....,e^{2\pi\sqrt{-1}x_u})$ for $(x_1,..., x_u)\in\mathbb{R}^u$.  We identify the tangent space of $\mathbb{T}^u$ at $e$(the identity element) with the tangent space of $\mathbb{R}^u$ at the origin. Let $l$ be the geodesic on $\mathbb{T}^u$ given by the image of $\sigma_t$. Let $Y$ be the tangent of $l$ at $e$. Write $Y = \sum\limits_{j=1}^ua_je_j$, where $e_i = (0, 0, .., 0, 1, 0, ..., 0)\in T_e\mathbb{T}^u, a_i\in\mathbb{R}$. By theorem \ref{thm-10} and that $l$ is dense in $\mathbb{T}^u$, we find $$g_i(x_1,..., x_u) = g_i(e)exp(\sqrt{-1}\sum\limits_{j=1}^ub_{ij}x_j).$$ Here $b_{ij}=2\pi k_{ij}$ for $k_{ij}\in\mathbb{Z}$. Restricting $g_i$ on $l$, we find $$d_i = 2\pi\sum\limits_{j=1}^u k_{ij}a_j.$$ This completes the proof of the corollary.

\end{proof}
\begin{cor}
Given the notations and assumptions as in corollary \ref{cor1},  let $w$ be the dimension of the isometry group of $\Sigma$. Then $w\geq v$.\end{cor}

\end{document}